\documentclass{amsart}

\usepackage{amsmath,amsthm,amsfonts,amssymb,amscd,latexsym}
\usepackage[all]{xy}

\newtheorem{thm}{Theorem}[section]
\newtheorem{lem}[thm]{Lemma}
\newtheorem{pro}[thm]{Proposition}
\newtheorem{cor}[thm]{Corollary}

\DeclareMathOperator{\id}{id}
\DeclareMathOperator{\Ker}{Ker}
\DeclareMathOperator{\im}{Im}
\DeclareMathOperator{\Hom}{Hom}
\DeclareMathOperator{\End}{End}
\DeclareMathOperator{\ext}{Ext}
\DeclareMathOperator{\ann}{Ann}
\DeclareMathOperator{\jac}{Jac}

\newcommand{\D}{\mathbb{D}}
\newcommand{\F}{\mathbb{F}}

\newcommand{\argu}{\hbox to 7truept{\hrulefill}}

\begin{document}


\title[Split strongly abelian $p$-chief factors and restricted cohomology]{Split strongly
abelian $p$-chief factors and first degree restricted cohomology}

\author{J\"org Feldvoss}
\address{Department of Mathematics and Statistics, University of South Alabama,
Mobile, AL 36688--0002, USA}
\email{jfeldvoss@southalabama.edu}

\author{Salvatore Siciliano}
\address{Dipartimento di Matematica e Fisica ``Ennio De Giorgi", Universit\`a del Salento,
Via Provinciale Lecce-Arnesano, I-73100 Lecce, Italy}
\email{salvatore.siciliano@unisalento.it}

\author{Thomas Weigel}
\address{Dipartimento di Matematica e Applicazioni, Universit\`a degli Studi di Milano-Bicocca,
Via Roberto Cozzi, No.\ 53, I-20125 Milano, Italy}
\email{thomas.weigel@unimib.it}




\begin{abstract}
In this paper we investigate the relation between the multiplicities of split strongly
abelian $p$-chief factors of finite-dimensional restricted Lie algebras and first
degree restricted cohomology. As an application we obtain a characterization of
solvable restricted Lie algebras in terms of the multiplicities of split strongly abelian
$p$-chief factors. Moreover, we derive some results in the representation theory
of restricted Lie algebras related to the principal block and the projective cover of
the trivial irreducible module of a finite-dimensional restricted Lie algebra. In particular,
we obtain a characterization of finite-dimensional solvable restricted Lie algebras in
terms of the second Loewy layer of the projective cover of the trivial irreducible
module. The analogues of these results are well known in the modular representation
theory of finite groups.
\end{abstract}

      
\date{January 21, 2013}
          
\maketitle


\section{Introduction} 


Let $p$ be an arbitrary prime number, and let $G$ be a finite group whose order
is divisible by $p$. Moreover, let $\F_p[G]$ denote the group algebra of $G$ over
the field $\F_p$ with $p$ elements, and let $S$ be an irreducible (unital left) $\F_p
[G]$-module. Then $[G:S]_{\rm p-split}$ denotes the number of {\em $p$-elementary
abelian chief factors\/} or for short {\em $p$-chief factors\/} $G_j/G_{j-1}$ ($1\le j
\le n$) of a given chief series $\{1\}=G_0\subset G_1\subset\cdots\subset G_n=G$
that are isomorphic to $S$ as $\F_p[G]$-modules and for which the exact sequence
$\{1\}\to G_j/G_{j-1}\to G/G_{j-1}\to G/G_j\to\{1\}$ splits in the category of groups.
It is well known that $[G:S]_{\rm p-split}$ is independent of the choice of the chief
series of $G$ (see also Theorem \ref{stammbach} below). 

W.\ Gasch\"utz proved the ``only if"-part of the following result on split (or complementable)
$p$-chief factors of finite $p$-solvable groups (see \cite[Theorem VII.15.5]{HB}).
The converse of Gasch\"utz' theorem is due to U.\ Stammbach \cite[Corollary 1]{S2}),
but in an equivalent form it was already proved earlier by W.\ Willems \cite[Theorem~3.9]{W}.

\begin{thm}\label{gaschuetz}
A finite group $G$ is $p$-solvable if, and only if,  $\dim_{\F_p}H^1(G,S)=\dim_{\F_p}
\End_{\F_p[G]}(S)\cdot [G:S]_{\rm p-split}$ holds for every irreducible $\F_p[G]$-module
$S$.
\end{thm}

Let $C_G (M):=\{g\in G\mid g\cdot m=m\mbox{ for every }m\in M\}$ denote the {\em
centralizer\/} of an $\F_p[G]$-module $M$ in $G$. In order to be able to apply his
cohomological characterization of $p$-solvable groups (see \cite[Theorem A]{S1}) in
the proof of Theorem \ref{gaschuetz}, Stammbach established the following result
(see the main result of \cite{S2}):

\begin{thm}\label{stammbach}
Let $G$ be a finite group, and let $S$ be an irreducible $\F_p[G]$-module with
centralizer algebra $\D:=\End_{\F_p[G]}(S)$. Then
\begin{equation*}
[G:S]_{\rm p-split}=\dim_\D H^1(G,S)-\dim_\D H^1(G/C_G(S),S)
\end{equation*}
holds. In particular, $[G:S]_{\rm p-split}$ is independent of the choice of the chief series
of $G$.
\end{thm}

The goal of this paper is to investigate whether analogues of Theorem \ref{gaschuetz}
and Theorem \ref{stammbach} hold in the context of restricted Lie algebras. Recently,
the authors have obtained analogues of these results for ordinary Lie algebras (see
the equivalence (i)$\Longleftrightarrow$(iii) in \cite[Theorem 4.3]{FSW} and \cite[Theorem
2.1]{FSW}, respectively). The main result of this paper is a restricted analogue of Theorem
\ref{stammbach} (see Theorem \ref{pabsplit}) from which all the other results will follow.
The proof given here follows the argument used in the proof of \cite[Theorem 2.1]{FSW}
very closely. An important tool in the proof is a one-to-one correspondence between the
set of equivalence classes of restricted extensions of a strongly abelian restricted Lie algebra
$M$ by a restricted Lie algebra $L$ acting on $M$ and the second restricted cohomology
space $H_*^2(L,M)$ that is defined via the transgression in the five-term exact sequence
associated to the restricted extension of $M$ by $L$.

As a consequence of Theorem~\ref{pabsplit} and the equivalence (i)$\Longleftrightarrow$(iv)
in \cite[Theorem 5.5]{FSW}, we obtain the analogue of Theorem \ref{gaschuetz} for split
strongly abelian $p$-chief factors of restricted Lie algebras (see Theorem \ref{charsolv}).
In the final section we apply the results obtained in Section 2 to the second Loewy layer
of the projective cover of the trivial irreducible module. The equivalence (i)$\Longleftrightarrow$(ii)
in Theorem \ref{LLPIM} is an analogue of Willems' module-theoretic characterization of
$p$-solvable groups (see \cite[Theorem 3.9]{W} and also \cite[Corollary~2]{S2}) for
restricted Lie algebras.

Let $\langle X\rangle_{\F}$ denote the $\F$-subspace of a vector space $V$ over a field
$\F$ spanned by a subset $X$ of $V$. For more notation and some well-known results
from the structure and representation theory of restricted Lie algebras we refer the reader
to Chapters 2 and 5 in \cite{SF}.


\section{Split strongly abelian $p$-chief factors and restricted cohomology}


In analogy to group theory we define a {\em $p$-chief series\/} for a finite-dimensional
restricted Lie algebra $L$ to be an ascending chain $0=L_0\subset L_1\subset\cdots
\subset L_n=L$ of $p$-ideals in $L$ such that $L_j/L_{j-1}$ is a minimal (non-zero)
$p$-ideal of $L/L_{j-1}$ for every integer $j$ with $1\le j\le n$. Any $L_j/L_{j-1}$ is
then called a {\em $p$-chief factor\/} of $L$. We say that $L_j/L_{j-1}$ is a {\it
strongly abelian $p$-chief factor\/} if it is an abelian Lie algebra with zero $p$-mapping
(see \cite[p.\ 565]{H} for the notion of a strongly abelian restricted Lie algebra).

Observe that strongly abelian $p$-chief factors are irreducible restricted modules but
this is not the case for arbitrary $p$-chief factors. Let $S$ be a simple Lie algebra that
is not restrictable, and let $L$ be the minimal $p$-envelope of $S$. Then $L$ has no
non-zero proper $p$-ideals (see \cite[Proposition 1.4(1)]{F3}), and therefore $L$ is
a $p$-chief factor of $L$ which is not irreducible as an $L$-module, because $S$ is a
non-zero proper $L$-submodule of $L$ (see \cite[Proposition 1.1(1)]{F3}). Note also
that every $p$-chief factor of a solvable restricted Lie algebra is abelian but not
necessarily strongly abelian as any non-zero torus shows.

For an irreducible $L$-module $S$ and a given $p$-chief series $0=L_0\subset L_1
\subset\cdots\subset L_n=L$ of $L$ we denote by $[L:S]_{\rm p-split}$ the number
of strongly abelian $p$-chief factors $L_j/L_{j-1}$ that are isomorphic to $S$ as an
$L$-module and for which the exact sequence $0\to L_j/L_{j-1}\to L/L_{j-1}\to L/L_j
\to 0$ splits in the category of restricted Lie algebras.  Since we will show in
Theorem~\ref{pabsplit} that $[L:S]_{\rm p-split}$ is independent of the choice of
the $p$-chief series, we will not indicate the $p$-chief series in the notation.

Let $L$ be a finite-dimensional restricted Lie algebra over a field $\F$ of prime
characteristic, and let $u(L)$ denote the restricted universal enveloping algebra
of $L$ (see \cite[p.\ 192]{J} or \cite[p.\ 90]{SF}). Then every restricted $L$-module
is an $u(L)$-module and vice versa, and so there is a bijection between the irreducible
restricted $L$-modules and the irreducible $u(L)$-modules. In particular, as $u(L)$
is finite-dimensional (see \cite[Theorem 12, p.\ 191]{J} or \cite[Theorem 2.5.1(2)]{SF}),
every irreducible restricted $L$-module is finite-dimensional. Following Hochschild
\cite{H}  we define the {\em restricted cohomology\/} of $L$ with coefficients in a
restricted $L$-module $M$ by $H_*^n(L,M):=\ext_{u(L)}^n(\F,M)$ for every
non-negative integer $n$.

In \cite[Section 3]{H}, Hochschild  discusses restricted extensions of a strongly
abelian restricted Lie algebra $M$ by a restricted Lie algebra $L$ with a fixed
action of $L$ on $M$. In particular, he shows in \cite[Theorem 3.3]{H} that the
{\em set of equivalence classes of restricted extensions\/} of $M$ by $L$, which
he denotes by $\mathrm{ext}_*(M,L)$, is a vector space canonically isomorphic
to $H_*^2(L,M)$. In the following we indicate how to derive part of this result
in a different way by using the transgression $d_2^\mathcal{E}$ in the five-term
exact sequence
\begin{equation}\label{5termext}
0\to H_*^1(L,M)\to H_*^1(E,M)\to\Hom_L(M,M)\overset{d_2^\mathcal{E}}{\to}
H_*^2(L,M)\to H_*^2(E,M)
\end{equation}
associated to any restricted extension $\mathcal{E}:0\to M\to E\to L\to 0$ of $M$
by $L$. It is clear from the naturality of \eqref{5termext} that $d_2^\mathcal{E}$
only depends on the equivalence class $[\mathcal{E}]$ of the restricted extension
$E$. Hence one can define a map $\Delta:\mathrm{ext}_*(M,L)\to H_*^2(L,M)$ by
$\Delta([\mathcal{E}]):=d_2^\mathcal{E}(\id_M)$. As for ordinary Lie algebras one
has the following result (see \cite[Theorem VII.3.3]{HS}):

\begin{lem}\label{pext}
Let $L$ be a restricted Lie algebra, and let $M$ be a strongly abelian restricted
Lie algebra over a field of prime characteristic $p$. Furthermore, assume that $M$
is a restricted $L$-module. Then $\Delta:\mathrm{ext}_*(M,L)\to H_*^2(L,M)$
is a bijection. Moreover, the equivalence class of a restricted extension of $M$
by $L$ is mapped to zero if, and only if, its representatives are split.
\end{lem}

\begin{proof}
As we will not need the surjectivity in this paper, we only prove the injectivity of
$\Delta$ and that restricted extensions are mapped to zero if, and only if, they
are split. For more details we refer the reader to the proofs of the analogous results
for groups (see \cite[Section VI.10]{HS}).

Let $\mathcal{F}:0\to R\to F\to L\to 0$ of $L$ be any free presentation of the restricted
Lie algebra $L$, and let $\mathcal{E}:0\to M\to E\to L\to 0$ be any restricted extension
of $M$ by $L$. Then there exist restricted Lie algebra homomorphisms $\varphi:F\to E$
and $\rho:R\to M$ such that the diagram
\begin{equation}\label{commdiagr}
\xymatrix{
0 \ar[r] & R\ar[d]^{\rho} \ar[r] & F \ar[d]^{\varphi} \ar[r] & L\ar@{=}[d] \ar[r] & 0\\ 
0 \ar[r] & M \ar[r] & E \ar[r] & L \ar[r] & 0
}
\end{equation}
commutes.

Moreover, $\rho$ induces an $L$-module homomorphism $\psi:\overline{R}\to M$, where
$\overline{R}:=R/[R,R]+\langle R^{[p]}\rangle_\F$. Observe that $\psi$ is surjective if,
and only if, $\varphi$ is surjective.

The commutativity of the diagram \eqref{commdiagr} in conjunction with the naturality of
the five-term exact sequence yields that the diagram
\begin{equation*}
\xymatrix{
& H_\ast^1(E,M)\ar[d]^{\varphi^\ast} \ar[r] & \Hom_L(M,M) \ar[d]^{\psi^\ast}
\ar[r]^{d_2^\mathcal{E}} & H_\ast^2(L,M)\ar@{=}[d] \ar[r] & H_\ast^2(E,M)\ar[d]\\ 
& H_\ast^1(F,M) \ar[r]^{\tau} & \Hom_L(\overline{R}, M) \ar[r]^{d_2^\mathcal{F}}
& H_\ast^2(L,M)\ar[r] & 0
}
\end{equation*}
is commutative as well. In particular, we obtain that
\begin{equation}\label{delta}
\Delta([\mathcal{E}])=d_2^\mathcal{E}(\id_M)=d_2^\mathcal{F}(\psi^*(\id_M))=
d_2^\mathcal{F}(\psi)\,.
\end{equation}

We are ready to prove the injectivity of $\Delta$. Let $\mathcal{E}_1:0\to M
\stackrel{\iota_1}{\to}E_1\stackrel{\pi_1}{\to} L\to 0$ and $\mathcal{E}_2:0\to M
\stackrel{\iota_2}{\to}E_2\stackrel{\pi_2}{\to} L\to 0$ be two restricted extensions of $M$
by $L$, and suppose that $\Delta([\mathcal{E}_1])=\Delta([\mathcal{E}_2])$. Let $F:=F(E_1
\oplus E_2)$ denote the free restricted Lie algebra generated by the underlying vector space
of $E_1\oplus E_2$. Consider the free presentation $\mathcal{F}:0\to R\to F\to L\to 0$ of
the restricted Lie algebra $L$. Then the restricted Lie algebra homomorphisms $\varphi_j:F
\to E_j$ ($j=1,2$) are surjective, and the $L$-module homomorphisms $\psi_j: \overline{R}
\to M$ induced by $\rho_j$ ($j=1,2$) are surjective as well.
 
It follows from $\Delta([\mathcal{E}_1])=\Delta([\mathcal{E}_2])$ and (\ref{delta}) that $\psi_1-\psi_2\in\Ker(d_2^\mathcal{F})$. Thus there exists a restricted derivation $D:F\to
M$ such that $\psi_1-\psi_2=\tau (D)$. Put $$\varphi_2^\prime:=\varphi_2+\iota_2\circ
D\,.$$ It is easily seen that $\varphi_2^\prime:F\to E_2$ is a restricted Lie algebra homomorphism.
Consider the commutative diagram
\begin{equation*}
\xymatrix{
0 \ar[r] & R \ar[d]^{\rho_2^\prime} \ar[r] & F \ar[d]^{ \varphi_2^\prime} \ar[r] & L \ar@{=}[d] \ar[r] & 0\\ 
0 \ar[r] & M \ar[r] & E_2 \ar[r] & L \ar[r] & 0
}
\end{equation*} 
If $\psi_2^\prime:\overline{R}\to M$ denotes the map induced by $\rho_2^\prime$ on $\overline{R}$,
then it is clear that $\psi_2^\prime=\psi_2+\tau(D)=\psi_1$. In particular, $\psi_2^\prime$ is surjective,
and, in turn, so is $\varphi_2^\prime$. Moreover, as $\psi_1=\psi_2^\prime$, we have $\rho_1=
\rho_2^\prime$, which implies that $\Ker(\varphi_1)=\Ker(\varphi_2^\prime)$. Consequently, $E_1
\cong F/\Ker(\varphi_1)=F/\Ker(\varphi_2^\prime)\cong E_2$. Consider the map $\eta:E_2\to E_1$
defined by $\varphi_2^\prime(f)\mapsto\varphi_1(f)$ for every $f\in F$. Then $\eta$ is a restricted Lie
algebra homomorphism making the following diagram commutative:
\begin{equation*}
\xymatrix{
\mathcal{E}_1:0 \ar[r] & M \ar@{=}[d] \ar[r]^{\iota_2} & E_2 \ar[d]^{ \eta} \ar[r]^{\pi_2} & L\ar@{=}[d] \ar[r] & 0\\ 
\mathcal{E}_2:0 \ar[r] & M \ar[r]^{\iota_1} & E_1 \ar[r]^{\pi_1} & L \ar[r] & 0
}
\end{equation*}
We conclude that $[\mathcal{E}_1]=[\mathcal{E}_2]$, which finishes the proof of the injectivity
of $\Delta$.

Finally, let us prove that $\Delta$ maps split restricted extensions to zero. Because of the
injectivity of $\Delta$, split restricted extensions are the only ones that are mapped to
zero. Without loss of generality we can assume that $E$ is the semi-direct product $M
\rtimes L$ of $M$ and $L$. Recall that the $p$-mapping on $E$ is defined by $(m,x)^{[p]}
:= (x^{p-1}\cdot m,x^{[p]})$ for every $m\in M$ and every $x\in L$ (see \cite[p.\  572]{H}).
We have to show that $d_2^\mathcal{E}(\id_M)=0$, or by the exactness of the
corresponding five-term sequence, that there exists a restricted derivation $D$ from
$E$ to $M$ such that the restriction of $D$ to $M$ is the identity. It is straightforward
to check that $D(m,x):=m$ for every $m\in M$ and every $x\in L$ defines such a restricted
derivation.
\end{proof}

We are ready to prove a restricted analogue of \cite[Lemma 2]{Ba} which will be essential in
the proof of our main result (see Theorem \ref{pabsplit}).

\begin{lem}\label{trans}
Let $L$ be a finite-dimensional restricted Lie algebra over a field of prime characteristic $p$,
let $I$ be a minimal $p$-ideal of $L$ that is strongly abelian, and let $\mathcal{E}$ denote
the equivalence class of the restricted extension $0\to I\to L\to L/I\to 0$. Then the following
statements hold:
\begin{itemize}
\item[{\rm (a)}] If $\mathcal{E}$ splits, then the transgression $d_2^\mathcal{E}:\Hom_L(I,I)
                          \to H_*^2(L/I,I)$ is zero.
\item[{\rm (b)}] If $\mathcal{E}$ does not split, then the transgression $d_2^\mathcal{E}:
                          \Hom_L(I,I)\to H_*^2(L/I,I)$ is injective.
\end{itemize}
\end{lem}

\begin{proof}
(a): It follows from Lemma \ref{pext} that $d_2^\mathcal{E}(\id_I)=0$. As $d_2^\mathcal{E}$
is compatible with the action of $\D:=\Hom_L(I,I)$, this implies that $d_2^\mathcal{E}=0$

(b): By virtue of the $\D$-linearity of $d_2^\mathcal{E}$, it is enough to show that $\Ker
(d_2^\mathcal{E})=0$. According to Lemma \ref{pext}, $d_2^\mathcal{E}(\id_I)\not= 0$.
Then the claim follows from  $$\dim_\D\Ker(d_2^\mathcal{E})+\dim_\D\im(d_2^\mathcal{E})
=\dim_\D\Hom_L(I,I)=1\,.$$
\end{proof}

\noindent {\bf Remark.} If one ignores  in the proofs of Lemma \ref{pext} and Lemma \ref{trans}
the compatibility of the homomorphisms with the $p$-mappings, then one obtains conceptual
proofs of \cite[Lemma 1 and 2]{Ba}, respectively.
\vspace{.3cm}

A restricted Lie algebra $L$ over $\F$ is called {\em $p$-perfect\/} if $L=[L,L]+\langle L^{[p]}
\rangle_\F$. By virtue of \cite[Proposition 2.7]{F1}, $L$ is $p$-perfect if, and only if, $H_*^1
(L,\F)=0$. Our main result is completely analogous to the main result of \cite{S2} (see also
\cite[Theorem 2.1]{FSW} for the analogue for ordinary Lie algebras):

\begin{thm}\label{pabsplit}
Let $L$ be a finite-dimensional restricted Lie algebra over a field of prime characteristic $p$,
and let $S$ be an irreducible $L$-module with centralizer algebra
$\D:=\End_L(S)$. Then
\begin{equation}\label{mult}
[L:S]_{\rm p-split}=\dim_\D H_*^1(L,S)-\dim_\D H_*^1(L/\ann_L(S),S)
\end{equation}
holds. In particular, $[L:S]_{\rm p-split}$ is independent of the choice of the $p$-chief
series of $L$.
\end{thm}

\begin{proof}
We proceed by induction on the dimension of $L$. If $L$ is one-dimensional, then $L$ is
either a torus or strongly abelian. For a torus both sides of \eqref{mult} vanish and in the
strongly abelian case the only irreducible restricted $L$-module is trivial so that both sides
of \eqref{mult} are also equal. Thus we may assume that the dimension of $L$ is greater
than one, and that the claim holds for all restricted Lie algebras of dimension less than
$\dim_\F L$. Let $0=L_0\subset L_1\subset\cdots\subset L_n=L$ be a $p$-chief series
of $L$. For the remainder of the proof the multiplicity $[L:S]_{p-{\rm split}}$ always refers
to this fixed $p$-chief series.

If $\ann_L(S)=0$, then the right-hand side of \eqref{mult} is zero. But as strongly abelian
$p$-chief factors have non-zero annihilators, the left-hand side also vanishes and the
assertion holds. So we may assume that $\ann_L(S)\neq 0$.

We first assume that $L_1\subseteq\ann_L(S)$. Then the five-term exact sequence
for restricted cohomology in conjunction with \cite[Proposition 2.7]{F1} yields the
exactness of
\begin{equation}\label{5term}
\begin{aligned}
0\longrightarrow H_*^1(L/L_1,S)&\longrightarrow H_*^1(L,S)\\
&\longrightarrow\Hom_L(L_1/[L_1,L_1]+\langle L_1^{[p]}\rangle_\F,S)\longrightarrow
H_*^2(L/L_1,S)\,.
\end{aligned}
\end{equation}
Since $S$ is also an irreducible restricted $L/L_1$-module, one obtains by
induction that
\begin{equation}\label{indhyp}
[L/L_1:S]_{\rm p-split}=\dim_\D H_*^1(L/L_1,S)-\dim_\D H_*^1(L/\ann_L(S),S)\,.
\end{equation}
As $L_1$ is a minimal $p$-ideal of $L$, $L_1$ is either $p$-perfect or strongly
abelian. In the former case, the third term in \eqref{5term} vanishes, and thus
$H_*^1(L/L_1,S)\cong H_*^1(L,S)$. Since $L_1$ is not strongly abelian, one
has $[L:S]_{\rm p-split}=[L/L_1:S]_{\rm p-split}$. Hence \eqref{mult} holds in
this case.

If $L_1$ is strongly abelian, one has $\Hom_L(L_1/[L_1,L_1]+\langle L_1^{[p]}
\rangle_\F,S)=\Hom_L(L_1,S)$. If in addition $L_1$ and $S$ are not isomorphic
as $L$-modules, then $\Hom_L(L_1,S)=0$, and the assertion follows as before.

For $L_1\cong S$ one has to distinguish two cases depending on the strongly
abelian $p$-chief factor $L_1$ being split, or being not split. In case that $L_1$
is split, one has
\begin{equation}\label{add}
\begin{aligned}
&&[L:S]_{\rm p-split} & =[L/L_1:S]_{\rm p-split}+1\\
&&& = \dim_\D H_*^1(L/L_1,S)-\dim_\D H_*^1(L/\ann_L(S),S)+1\,,
\end{aligned}
\end{equation}
and Lemma \ref{trans}(a) shows that the transgression $\Hom_L(L_1,S)\to H_*^2
(L/L_1,S)$ is zero. Thus the exactness of \eqref{5term} implies that the restriction
$H_*^1(L,S)\to\Hom_L(L_1,S)$ is surjective, and therefore
\begin{equation}\label{add2}
\begin{aligned}
&& \dim_\D H_*^1(L,S) & =\dim_\D H_*^1(L/L_1,S)+\dim_\D\Hom_L(L_1,S)\\
&&& =\dim_\D H_*^1(L/L_1,S)+1\,.
\end{aligned}
\end{equation}
Hence \eqref{add} and \eqref{add2} yield the assertion. Suppose now that $L_1$
is not split. In this case Lemma \ref{trans}(b) implies that the transgression $\Hom_L
(L_1,S)\to H_*^2(L/L_1,S)$ is injective. According to \eqref{5term}, the inflation
$H_*^1(L/L_1,S)\to H_*^1(L,S)$ is bijective. Then one has $[L:S]_{\rm p-split}=
[L/L_1:S]_{\rm p-split}$, and the claim follows from \eqref{indhyp}.

Finally, assume that $L_1\not\subseteq\ann_L(S)$, i.e., $L_1\cap\ann_L(S)=0$ and
$S^{L_1}=0$. Suppose that $L_j/L_{j-1}$ is strongly abelian and $L_j/L_{j-1}\cong
S$ as $L$-modules for some integer $j$ with $1\leq j\leq n$. Then $L_j$ -- and thus
$L_1$ -- would be contained in $\ann_L(S)$, a contradiction. Hence $[L:S]_{\rm p-split}
=0$. As $S^{L_1}=0$, one concludes from the beginning of the five-term exact sequence
\begin{equation*}
0\longrightarrow H_*^1(L/L_1,S^{L_1})\longrightarrow H_*^1(L,S)\longrightarrow
H_*^1(L_1,S)^L\longrightarrow H_*^2(L/L_1,S^{L_1})
\end{equation*}
that the vertical mappings in the commutative diagram
\begin{equation*}
\label{eq:comdia}
\xymatrix{
H_*^1(L/\ann_L(S),S)\ar[r]^-\alpha\ar[d]&H_*^1(L,S)\ar[d]\\
H_*^1(L_1+\ann_L(S)/\ann_L(S),S)^L\ar[r]^-\beta&H_*^1(L_1,S)^L
}
\end{equation*}
are isomorphisms. Because $\beta$ is an isomorphism, $\alpha$ is an isomorphism as
well. This shows that in this case the right-hand side of \eqref{mult} is also zero.

Since the right-hand side of \eqref{mult} does not depend on the choice of the $p$-chief
series, the left-hand side does not either. This completes the proof of the theorem.
\end{proof}

In the extreme case $\ann_L(S)=L$, Theorem \ref{pabsplit} in conjunction with \cite[Proposition
2.7]{F1} has the following consequence:

\begin{cor}\label{triv}
Let $L$ be a finite-dimensional restricted Lie algebra over a field $\F$ of prime characteristic $p$.
Then the trivial irreducible $L$-module occurs with multiplicity $\dim_\F L/[L,L]+\langle L^{[p]}
\rangle_\F$ as a split strongly abelian $p$-chief factor of $L$.
\end{cor}

Moreover, the next result follows from Hochschild's six-term exact sequence relating ordinary
and restricted cohomology (see \cite[p.\ 575]{H}) in conjunction with Corollary \ref{triv} and
\cite[Corollary 2.2]{FSW}. (Here $[L:S]_{\rm split}$ denotes the multiplicity of $S$ as a split
abelian chief factor of the ordinary Lie algebra $L$.)

\begin{cor}\label{split}
Let $L$ be a finite-dimensional restricted Lie algebra over a field $\F$ of prime characteristic
$p$. If $S$ is an irreducible restricted $L$-module, then
\begin{eqnarray*}
\lefteqn{[L:S]_{\rm p-split}}\\
&& =\left\{\begin{array}{ll}
[L:S]_{\rm split}&\mbox{\rm if }S\not\cong\F\\
{}[L:S]_{\rm split}-\dim_\F(\langle L^{[p]}\rangle_\F/[L,L]\cap\langle L^{[p]}\rangle_\F)
&\mbox{\rm if }S\cong\F
\end{array}\right. .
\end{eqnarray*}
In particular, $[L:S]_{\rm p-split}\le[L:S]_{\rm split}$.
\end{cor}

\noindent The equality of $[L:S]_{\rm p-split}$ and $[L:S]_{\rm split}$ for non-trivial irreducible
restricted $L$-modules $S$ explains why the results in Section 5 of \cite{FSW} could be obtained
although their ingredients belong to different categories. 

Recall that the {\em principal block\/} of a restricted Lie algebra is the block that
contains the trivial irreducible module. For the convenience of the reader we include
a proof of the following result which is completely analogous to the corresponding
proof for modular group algebras (see \cite[Proposition 1]{S1}).

\begin{pro}\label{pchiefpriblo}
Every strongly abelian $p$-chief factor of a finite-dimensional restricted Lie algebra $L$ belongs
to the principal block of $L$.
\end{pro}

\begin{proof}
Let $S=I/J$ be a strongly abelian $p$-chief factor of $L$. In particular, $S$ is a
trivial $I$-module. Then the five-term exact sequence for restricted cohomology in
conjunction with \cite[Proposition 2.7]{F1} yields the exactness of $$0\longrightarrow
H_*^1(L/I,S)\longrightarrow H_*^1(L/J,S)\longrightarrow\Hom_L(S,S)\longrightarrow
H_*^2(L/I,S)\,.$$ Since the third term is non-zero, the second or fourth term must
also be non-zero. According to \cite[Lemma 1(a)]{F2}, in either case $S$ belongs
to the principal block of a restricted Lie factor algebra of $L$. Then it follows from
\cite[Lemma 4]{F2} that $S$ also belongs to the principal block of $L$.
\end{proof}

The analogue of Theorem \ref{gaschuetz} for restricted Lie algebras is another consequence
of Theorem \ref{pabsplit} in conjunction with the equivalence (i)$\Longleftrightarrow$(iv) in
\cite[Theorem 5.5]{FSW}.

\begin{thm}\label{charsolv}
Let $L$ be a finite-dimensional restricted Lie algebra over a field $\F$ of prime characteristic
$p$. Then the following statements are equivalent:
\begin{enumerate}
\item [(i)]   $L$ is solvable.
\item [(ii)]  $\dim_\F H_*^1(L,S)=\dim_\F\End_L(S)\cdot[L:S]_{\rm p-split}$ holds for every
                  irreducible $L$-module $S$.
\item [(iii)] $\dim_\F H_*^1(L,S)=\dim_\F\End_L(S)\cdot[L:S]_{\rm p-split}$ holds for every
                  irreducible $L$-module $S$ belonging to the principal block of $L$.
\end{enumerate}
\end{thm}

\begin{proof}
The equivalence of (i) and (ii) is a consequence of Theorem \ref{pabsplit} and the equivalence
(i)$\Longleftrightarrow$(iv) in \cite[Theorem 5.5]{FSW}, and the equivalence of (ii) and (iii)
follows from \cite[Lemma 1(a)]{F2} in conjunction with Proposition \ref{pchiefpriblo}.
\end{proof}

\noindent {\bf Remark.} It is an immediate consequence of Theorem \ref{pabsplit}
that $\dim_\F H_*^1(L,S)=\dim_\F\End_L(S)\cdot[L:S]_{\rm p-split}$ holds for the
trivial irreducible $L$-module $S$. Hence one can also obtain Theorem \ref{charsolv}
immediately from Corollary \ref{split} and the equivalence of (i), (vi), and (vii) in
\cite[Theorem 5.5]{FSW}. We included the proof given above since it is the precise
analogue of the proof of \cite[Corollary 1]{S2}.


\section{Split strongly abelian $p$-chief factors and the $0$-PIM}


Let $A$ be a finite-dimensional (unital) associative algebra with Jacobson radical
$\jac(A)$, and let $M$ be a (unital left) $A$-module. Then the descending filtration
$$M\supset\jac(A)M\supset\jac(A)^2M\supset\jac(A)^3M\supset\cdots\supset\jac
(A)^{\ell}M\supset\jac(A)^{\ell+1}M=0$$ is called the {\em Loewy series\/} of $M$
and the factor module $\jac(A)^{n-1}M/\jac(A)^nM$  is called the $n^{\rm th}$ {\em
Loewy layer\/} of $M$ (see \cite[Definition 1.2.1]{B} or \cite[Definition VII.10.10a)]{HB}).

Recall that a projective module $P_A(M)$ is a {\it projective cover\/} of $M$, if there
exists an $A$-module epimorphism $\pi_M$ from $P_A(M)$ onto $M$ such that
the kernel of $\pi_M$ is contained in the radical $\jac(A)P_A(M)$ of $P_A(M)$. If
projective covers exist, then they are unique up to isomorphism. It is well known that
projective covers of finite-dimensional modules over finite-dimensional associative
algebras exist and are again finite-dimensional. Moreover, every projective
indecomposable $A$-module is isomorphic to the projective cover of some irreducible
$A$-module. In this way one obtains a bijection between the isomorphism classes
of the projective indecomposable $A$-modules and the isomorphism classes of the
irreducible $A$-modules.

In the sequel we use the notation $P_L(\F):=P_{u(L)}(\F)$ for the projective cover
of the trivial irreducible module of a finite-dimensional restricted Lie algebra $L$ over
a field $\F$ of prime characteristic. Using \cite[Proposition 2.4.3]{B} and Theorem
\ref{pabsplit} we obtain a lower bound for the multiplicity $[\jac(u(L))P_L(\F)/\jac
(u(L))^2P_L(\F):S]$ of an irreducible restricted $L$-module $S$ in the second Loewy
layer of $P_L(\F)$ (see \cite[Theorem 3.7]{W} for the analogue in the modular
representation theory of finite groups):

\begin{thm}\label{loewybd}
Let $L$ be a finite-dimensional restricted Lie algebra over a field $\F$ of prime
characteristic $p$. Then $$[\jac(u(L))P_L(\F)/\jac(u(L))^2P_L(\F):S]\ge[L:S]_{\rm
p-split}$$ for every irreducible restricted $L$-module $S$.
\end{thm}

\begin{proof}
We obtain from \cite[Proposition 2.4.3]{B} and Theorem \ref{pabsplit} that
\begin{align*}
\dim_\F\End_L(S)&\cdot[\jac(u(L))P_L(\F)/\jac(u(L))^2P_L(\F):S]\\
&=\dim_\F\ext_{u(L)}^1(\F,S)=\dim_\F H_*^1(L,S)\\
&\ge\dim_\F H_*^1(L,S)-\dim_\F H_*^1(L/\ann_L(S),S)\\
&=\dim_\F\End_L(S)\cdot[L:S]_{\rm p-split}\,.
\end{align*}
Cancelling $\dim_\F\End_L(S)$ yields the desired inequality.
\end{proof}

\noindent {\bf Remark.} If one uses the main result of \cite{S2} instead of Theorem
\ref{pabsplit}, then the above proof would also work in the case of finite-dimensional
modular group algebras. This provides an alternative proof of \cite[Theorem 3.7]{W}.
\vspace{.3cm}

The following example shows that equality does not necessarily hold in Theorem~\ref{loewybd}.
We will see soon that equality holds if, and only if, the restricted Lie algebra is solvable
(see the equivalence (i)$\Longleftrightarrow$(ii) in Theorem \ref{LLPIM}).
\vspace{.3cm}

\noindent {\bf Example.} Consider the three-dimensional restricted simple Lie algebra
$L:=\mathfrak{sl}_2(\F)$ over an algebraically closed field $\F$ of characteristic $p>2$.
Take for $S$ the $(p-1)$-dimensional irreducible restricted $L$-module. Then it follows
from \cite[Theorem~1(ii)]{P} that $[\jac(u(L))P_L(\F)/\jac(u(L))^2P_L(\F):S]=2$, but
$[L:S]_{\rm p-split}=0$.
\vspace{.0001cm}

As an immediate consequence of Theorem \ref{loewybd}, we obtain the following weak
analogue of a well-known result for finite modular group algebras:

\begin{cor}\label{psplitsolv}
Every split strongly abelian $p$-chief factor of a finite-dimensional restricted Lie algebra
$L$ is a direct summand of the second Loewy layer of the projective cover $P_L(\F)$ of the
trivial irreducible $L$-module. In particular, every split strongly abelian $p$-chief factor of a
finite-dimensional restricted Lie algebra $L$ is a composition factor of $P_L(\F)$.
\end{cor}

\noindent {\bf Question.} In view of Corollary \ref{psplitsolv}, it is natural to ask whether
every strongly abelian $p$-chief factor of a finite-dimensional solvable restricted Lie
algebra $L$ is a composition factor of $P_L(\F)$, or even more generally (see Proposition
\ref{pchiefpriblo}), whether every irreducible module in the principal block of $u(L)$ is a
composition factor of $P_L(\F)$  (for an affirmative answer to the analogous question in
the modular representation theory of finite $p$-solvable groups see \cite[Theorem VII.15.8]{HB}).
\vspace{.3cm}

Finally, we obtain the following characterization of solvable restricted Lie algebras which
was motivated by \cite[Theorem 5.5]{FSW} but contrary to the latter allows to include
the trivial irreducible module in the implications (i)$\Longrightarrow$(ii) and (i)$\Longrightarrow$(iii)
(see \cite[Theorem 3.9]{W} for the analogue of (i)$\Longleftrightarrow$(ii) in the modular
representation theory of finite groups).

\begin{thm}\label{LLPIM}
Let $L$ be a finite-dimensional restricted Lie algebra over a field $\F$ of prime characteristic
$p$. Then the following statements are equivalent:
\begin{enumerate}
\item [(i)]   $L$ is solvable.
\item [(ii)]  $[\jac(u(L))P_L(\F)/\jac(u(L))^2P_L(\F):S]=[L:S]_{\rm p-split}$ for every irreducible
                  restricted $L$-module $S$.
\item [(iii)] $[\jac(u(L))P_L(\F)/\jac(u(L))^2P_L(\F):S]=[L:S]_{\rm p-split}$ for every irreducible
                  restricted $L$-module $S$ belonging to the principal block of $L$.
\end{enumerate}
\end{thm}

\begin{proof}
The equivalence of the three statements  is a consequence of Theorem \ref{charsolv}  in conjunction
with $$\dim_\F\End_L(S)\cdot[\jac(u(L))P_L(\F)/\jac(u(L))^2P_L(\F):S]=\dim_\F H_*^1(L,S)\,.$$
\end{proof}

\noindent {\bf Remark.} It follows from the proof of Theorem \ref{loewybd} that the equality
in statements (ii) and (iii) of Theorem \ref{LLPIM} holds for the trivial irreducible $L$-module.
Hence one can also obtain Theorem \ref{LLPIM} from Corollary \ref{split} and the equivalence
of (i), (viii), and (ix) in \cite[Theorem 5.5]{FSW}.
\vspace{.3cm}

\noindent {\bf Acknowledgements.} The first and the second author would like to thank the Dipartimento
di Matematica e Applicazioni at the Universit\`a degli Studi di Milano-Bicocca for the hospitality during
their visit in May 2012 when parts of this paper were written.



\end{document}